\documentclass{birkjour}
\usepackage{amsmath,amsthm,amssymb}

\newtheorem{theorem}{Theorem}[section]
 \newtheorem{corollary}{Corollary}[section]

\newtheorem{lemma}{Lemma}[section]

\DeclareMathOperator{\RM}{Re}
\allowdisplaybreaks
\begin{document}

\title[On the Bessel function]{Inclusion of generalized Bessel functions in the Janowski class}
\author[S. R. Mondal]{Saiful R. Mondal}
\address{ Department of Mathematics and Statistics, College of Science,
King Faisal University, Al-Hasa 31982, Saudi Arabia }
\email{smondal@kfu.edu.sa}
\author[Al Dhuain Mohammed]{Al Dhuain Mohammed}
\address{ Department of Mathematics and Statistics, College of Science,
King Faisal University, Al-Hasa 31982, Saudi Arabia}
\email{albhishi1900@gmail.com}

\begin{abstract}
Sufficient conditions on $A$, $B$, $p$, $b$ and $c$ are determined that will ensure the generalized Bessel functions  ${u}_{p,b,c}$ satisfies the subordination ${u}_{p,b,c}(z) \prec (1+Az)/ (1+Bz)$. In particular this gives  conditions for $(-4\kappa/c)({u}_{p,b,c}(z)-1)$, $c \neq 0$ to be close-to-convex. Also, conditions for which  ${u}_{p,b,c}(z)$ to be Janowski convex, and $z{u}_{p,b,c}(z)$ to be Janowski starlike in the unit disk $\mathbb{D}=\{z \in \mathbb{C}: |z|<1\}$ are obtained.
\end{abstract}
\maketitle

\section{Introduction}
Let $\mathcal{A}$ denote the class of analytic functions $f$ defined in the open unit disk $\mathbb{D}=\{z: |z|<1\}$ normalized by the conditions
 $f(0) = 0 = f'(0)-1$. If $f$ and $g$ are  analytic
in $\mathbb{D}$, then $f$ is subordinate to $g$, written  $f(z) \prec g(z)$,  if there is an analytic self-map
$w$ of $\mathbb{D}$ satisfying $w(0)=0$  and $f = g \circ w$.
For $-1 \leq B < A \leq  1$, let $\mathcal{P}[A,B]$ be the class consisting of normalized analytic functions $p(z)= 1+ c_{1}z + \cdots$ in $\mathbb{D}$ satisfying
\[ p(z) \prec \frac{1+Az}{1+Bz}.\]
For instance, if $0 \leq \beta <1$, then $ \mathcal{P}[1-2 \beta, -1]$ is the class of functions $p(z)= 1+ c_{1}z + \cdots$ satisfying $ \RM p(z) > \beta$ in $\mathbb{D}$.

The class $\mathcal{S}^\ast[A, B]$ of Janowski starlike functions \cite{Janowski} consists of $ f \in \mathcal{A}$ satisfying
\[\frac{ z f'(z)}{f(z)}\in \mathcal{P}[A,B].\]
For $0 \leq \beta <1$, $ \mathcal{S}^\ast[1-2 \beta, -1]:=  \mathcal{S}^\ast(\beta)$ is the usual class of starlike functions of order $\beta$;
$\mathcal{S}^\ast[1- \beta, 0]:=  \mathcal{S}^\ast_{\beta} = \{f \in \mathcal{A} :
| z f'(z)/f(z) -1| < 1- \beta \}$, and $\mathcal{S}^\ast[\beta, - \beta]:=\mathcal{S}^\ast[\beta]= \{f \in \mathcal{A} :
| z f'(z)/f(z) -1| < \beta | z f'(z)/f(z) +1|\}$. These classes have been studied, for example, in
\cite{Ali-seeni-ijmms, Ali-Chandra-aml}. A function $f \in \mathcal{A}$ is said to be close-to-convex of order
$\beta$ \cite{Goodman-book, Miller-Mocanu-book} if $\RM\left( zf'(z)/{g(z)}\right)> \beta$ for some
 $g \in \mathcal{S}^\ast :=\mathcal{S}^\ast(0)$.

This article studies the generalized  Beesel function $u_p(z)={u}_{p,b,c}(z)$ given by the power series
\begin{align}\label{eqn:conf-hyp}
u_p(z) = {}_0F_{1}(\kappa,\tfrac{-c}{4}z) =  \sum_{k=0}^\infty \tfrac{(-1)^k c^k}{4^k(\kappa)_k}\tfrac{z^k}{k!},
\end{align}
where $ \kappa=p+(b+1)/2\neq 0,-1,-2,-3 \cdots$. The function $u_p(z)$ is analytic in $\mathbb{D}$ and solution of the differential equation
\begin{align}\label{eqn:kumar-hypr-ode}
4z^2 u''(z)+4\kappa zu'(z)+c z u(z)=0,
\end{align}
if  $b$,$p$,$c$ in $\mathbb{C}$ ,such that $\kappa =p+(b+1)/2 \neq 0,-1,-2,-3 \cdots$ and   $z\in\mathbb{D}$. This normalized and generalized Bessel function of the first kind of order $p$, also satisfy the following recurrence relation
\begin{align}\label{eqn:kumar-hypr-recur-1}
4\kappa u'_p(z)=-c u_{p+1} (z),
\end{align}
which is an useful tool to study several geometric properties of $u_p$. There has been several works \cite{Baricz1,Baricz2,Selinger,Szasz-Kupan,Baricz-Pswamy,Baricz-Szasz} studying geometric properties of the function $u_p(z)$, such as on its close-to-convexity, starlikeness, and convexity, radius of starlikeness and convexity.

In Section $\ref{section-3}$  of this paper, sufficient conditions on $A$, $B$, $c$, $\kappa$ are determined  that will ensure $u_p$ satisfies the subordination $u_p(z) \prec (1+Az)/ (1+Bz)$. It is to be understood that a computationally-intensive methodology with shrewd manipulations is required to obtain the results in this general framework. The benefits of such general results are that by judicious choices of the parameters $A$ and $B$, they give rise to several interesting applications, which include extending the results of previous works. Using this subordination result, sufficient conditions are obtained for $(-4\kappa/c) u'(z) \in \mathcal{P}[ A, B]$, which next readily gives conditions for $(-4\kappa/c)( u_p(z)-1)$ to be close-to-convex. Section  $\ref{section-4}$ gives emphasis to the investigation of $u_p(z)$ to be Janowski convex as well as of $ z u_p(z)$ to be Janowski starlike.

The following lemma is needed in the sequel.
\begin{lemma}\cite{miller-dif-sub, Miller-Mocanu-book}\label{lem:miller-mocanu-1}
 Let $\Omega \subset \mathbb{C}$, and $\Psi : \mathbb{C}^2 \times \mathbb{D} \to \mathbb{C}$ satisfy
 \[
 \Psi( i \rho , \sigma; z) \not \in \Omega
 \]
 whenever $z \in \mathbb{D}$, $\rho$ real, $\sigma \leq -(1+\rho^2)/2$. If $p$ is analytic in $\mathbb{D}$ with $p(0)=1$, and
 $\Psi( p(z), z p'(z); z) \in \Omega$ for $ z \in \mathbb{D}$, then $\RM  p(z) > 0$ in $\mathbb{D}$.
\end{lemma}

In the case $\Psi : \mathbb{C}^3 \times \mathbb{D} \to \mathbb{C}$, then the condition in Lemma $\ref{lem:miller-mocanu-1}$ generalized to
 \[
 \Psi( i \rho , \sigma, \mu + i \nu; z) \not \in \Omega
 \]
 $\rho$ real,  $\sigma+\mu \leq 0$ and $ \sigma \leq -(1+\rho^2)/2$.

\section{Close-to-convexity of the Bessel function}\label{section-3}
In this section,  one main result on the close-to-convexity of the generalized Bessel function with several consequences are discussed in details.

\begin{theorem}\label{thm:-janw-cc}
Let $-1 \leq B \leq 3- 2 \sqrt{2} \approx 0.171573$. Suppose
 $B < A \leq 1$, and $c$, $\kappa  \in \mathbb{R}$ satisfy
 \begin{align}\label{eqn:thm-janw-cc-0}
\kappa-1 \geq \max \left\{0,  \tfrac{(1+B)(1+A)}{4(A-B)}\left|c\right|\right\}.
\end{align}
Further let $A$, $B$, $\kappa$ and $c$ satisfy either the inequality
\begin{align}\label{eqn:thm-janw-cc-2}
(\kappa-1)^2+ \tfrac{(\kappa-1)(1+B)}{(1-B)}- \left| \tfrac{(\kappa-1)(A+B) }{2(A-B)} c  + \tfrac{(1+B)^2(1+A) }{4(1-B)(A-B)} c\right|  \geq  \tfrac{(1-A^2) (1-B^2)}{16(A-B)^2} c^2
\end{align}
whenever
\begin{align}\label{eqn:thm-janw-cc-1}\left| 2 (\kappa-1)(1-B)(A+B) c+(1+B)^2(1+A) c \right|
 \geq \tfrac{1}{2}(A-B)(1-B) c^2,
\end{align}
or the inequality
\begin{align}\label{eqn:thm-janw-cc-3}\nonumber
&\left((\kappa-1)  \tfrac{(A+B)}{2(A-B)}c + \tfrac{(1+B)^2((1+A)}{4(1-B)(A-B)}c \right)^2
\\
&\quad \quad \leq \tfrac{c^2}{4} \left((\kappa-1)^2 +(\kappa-1)\tfrac{(1+B)}{1-B}  -\tfrac{(1-AB)^2}{16(A-B)^2}c^2 \right)
\end{align}
whenever
\begin{align}\label{eqn:thm-janw-cc-4}\left| 2 (\kappa-1)(1-B)(A+B)c + (1+B)^2(1+A) c \right|< \tfrac{1}{2}(A-B)(1-B) c^2.
\end{align}
If $(1+B)u_p (z) \neq (1+A)$,  then $u_p (z) \in \mathcal{P}[A,B]$.
\end{theorem}

\begin{proof}
Define the analytic function $p : \mathbb{D} \to  \mathbb{C}$ by
\begin{align*}
p(z) = - \tfrac{ (1-A) - (1-B) u_p(z)}{ (1+A) - (1+B) u_p(z)},
\end{align*}
 Then, a computation yields
\begin{align}\label{eqn-thm-1-1}
 u_p(z)& = \tfrac{(1-A) + (1+A) p(z)}{(1-B) + (1+B) p(z)},\\ \label{eqn-thm-1-2}
u_p(z)& = \tfrac{ 2 (A-B) p'(z)}{((1-B) + (1+B) p(z))^2 },
\end{align}
and
\begin{align}\label{eqn-thm-1-3}
 u_p''(z) =\tfrac{2 (A-B)( (1-B) + (1+B) p(z) ) p''(z)
- 4 (1+B) (A-B) {p'}^2(z)}{( (1-B) + (1+B) p(z) )^3}.
\end{align}
Thus, using the identities $(\ref{eqn-thm-1-1})$--$(\ref{eqn-thm-1-3})$, the Bessel differential equation
$(\ref{eqn:kumar-hypr-ode})$ can be rewrite as
\begin{align}\label{eqn:thm-1-ode} \nonumber
&z^2 p''(z) - \tfrac{ 2( 1+B)}{(1-B)+ (1+B)p(z)}  (zp'(z))^2 + \kappa zp'(z)\\
& \quad \quad \quad \quad + \tfrac{((1-B) + (1+B) p(z) ) ((1-A)+ (1+A)p(z))}{ 8 (A-B)} cz = 0.
\end{align}

Assume $\Omega = \{0\}$, and  define $\Psi(r, s, t;z)$ by
\begin{align}\label{eqn:thm-1-ode-2}
\Psi(r, s, t;z) := t &- \tfrac{ 2( 1+B)}{(1-B)+ (1+B)r} s^2 + \kappa s +
\tfrac{((1-B) + (1+B) r ) ((1-A)+ (1+A)r )}{ 8 (A-B)} cz.
\end{align}
It follows from  $(\ref{eqn:thm-1-ode})$ that
$
\Psi(p(z), z p'(z), z^2 p''(z); z)  \in \Omega.$  To ensure $\RM p(z) >0$ for $z \in \mathbb{D}$,
from Lemma $\ref{lem:miller-mocanu-1}$, it is enough to establish
$
\RM \Psi( i\rho, \sigma, \mu+ i \nu; z) < 0
$
in $\mathbb{D}$  for any real $\rho$, $\sigma \leq -(1+\rho^2)/2$, and $\sigma+\mu \leq 0$.

With $z = x+ iy \in \mathbb{D}$ in $(\ref{eqn:thm-1-ode-2})$, a computation yields
\begin{align}\label{eqn:thm-1-re-psi} \nonumber
\RM \Psi( i\rho, \sigma, \mu+ i \nu; z)
& = \mu - \tfrac{ 2( 1-B^2)}{(1-B)^2+ (1+B)^2 \rho^2} \sigma^2 + \kappa \sigma
 - \tfrac{\rho (1-AB)}{4(A-B)} cy\\
 &\quad \quad + \tfrac{(1-B) (1-A) - (1+B) (1+A)\rho^2}{8(A-B)} cx.
\end{align}
Since $\sigma \leq -(1+\rho^2)/2$, and $B \in [-1, 3-2\sqrt{2}]$,
\begin{align*}
\tfrac{ 2( 1-B^2)}{(1-B)^2+ (1+B)^2 \rho^2} \sigma^2
 &\geq \tfrac{ 2( 1-B^2)}{(1-B)^2+ (1+B)^2 \rho^2} \tfrac{(1+\rho^2)^2}{4}
 \geq \dfrac{1+B}{2(1-B)}.
\end{align*}
 Thus
\begin{align*}
\RM \Psi( i\rho, \sigma, \mu+ i \nu; z)&\leq (\kappa -1)\sigma - \tfrac{1+B}{2(1-B)} - \tfrac{\rho(1-AB)}{4(A-B)} cy\\
&\hspace{.5in}+\tfrac{(1-B)(1-A)-(1+B)(1+A)\rho^2}{8(A-B)}cx\\
&\leq -\tfrac{1}{2}(\kappa-1) (1+\rho^2) -\tfrac{1+B}{2(1-B)}-\tfrac{\rho(1-AB)}{4(A-B)}cy\\
&\hspace{.5in}+\tfrac{(1-B)(1-A)-(1+B)(1+A)\rho^2}{8(A-B)}cx\\
&=p_1 \rho^2+  q_1 \rho + r_1 := Q(\rho),
\end{align*}
where
\begin{align*}
p_1&= -\tfrac{1}{2}(\kappa-1)-\tfrac{(1+B) (1+A)  cx}{8(A-B)}, \hspace{2.5in} \\
q_1&=-\tfrac{1-AB}{4(A-B)}cy,\\
r_1 &= - \tfrac{1}{2} (\kappa-1)  +\tfrac{(1-B)(1-A)}{8(A-B)}cx-\tfrac{1+B}{2(1-B)}.
\end{align*}

Condition $(\ref{eqn:thm-janw-cc-0})$ shows that
\begin{align*}
p_1 &= -\tfrac{1}{2}(\kappa-1)-\tfrac{(1+B) (1+A)  cx}{8(A-B)}\\
 &< -\tfrac{1}{2}\bigg( (\kappa-1)-\tfrac{(1+B)(1+A)}{4(A-B)}\left|c\right|\bigg) <0.
\end{align*}
 Since $\displaystyle{ \max_{\rho \in \mathbb{R}} \{ p_1 \rho^2+ q_1\rho + r_1\}=(4 p_1 r_1 - q_1^2)/(4 p_1)}$ for $p_1 < 0$, it is clear that $Q(\rho) <0$ when
 \begin{align*}
\tfrac{(1-AB)^2}{16(A-B)^2} c^2y^2 &<  4\left(-\tfrac{1}{2}(\kappa-1)-\tfrac{(1+B)(1+A)}{8(A-B)}cx\right)\times\\ & \quad \quad \quad \quad \left(-\tfrac{1}{2}(\kappa-1)+\tfrac{(1-B)(1-A)}{8(A-B)}cx
-\tfrac{1+B}{2(1-B)}\right),
 \end{align*}
 with $|x|, |y| < 1$.
 As $y^2 < 1- x^2$, the above condition holds whenever
 \begin{align*}
&\tfrac{(1-AB)^2 c^2}{16(A-B)^2} (1-x^2)\\& \leq  \left( (\kappa-1)+\tfrac{(1+B) (1+A)  cx}{4(A-B)}\right) \left((\kappa-1)-\tfrac{(1-B) (1-A)}{4(A-B)} cx + \tfrac{1+B}{1-B}\right),
 \end{align*}
 that is, when
\begin{align}\label{eqn:thm-1-x}\nonumber
&\tfrac{c^2}{16}x^2 + \left((\kappa-1)\tfrac{(A+B)}{2(A-B)}c +\tfrac{(1+B)^2(1+A)}{4(1-B)(A-B)}c\right)x
  \\& \quad \quad + (\kappa-1)^2 + (\kappa-1)\tfrac{1+B}{1-B}- \tfrac{(1-AB)^2}{16(A-B)^2}c^2 \geq 0.
\end{align}
To establish inequality $(\ref{eqn:thm-1-x})$, consider the polynomial $R$ given by
\begin{align*}\label{eqn:thm-1-R(x)}
R(x) := m x^2 + n x+ r,\quad |x| <1,
\end{align*}
where
\begin{align*}
m &:=\tfrac{c^2}{16},\quad \quad \quad
n := (\kappa-1)\tfrac{(A+B)}{2(A-B)} c +\tfrac{(1+B)^2(1+A)}{4(1-B)(A-B)} c \\
r &:= (\kappa-1)^2 + (\kappa-1)\tfrac{1+B}{1-B}- \tfrac{(1-AB)^2}{16(A-B)^2}c^2.
\end{align*}
The constraint $(\ref{eqn:thm-janw-cc-1})$ yields $|n| \geq 2|m|$, and thus $R(x) \geq m +  r - |n|$. Now inequality
 $(\ref{eqn:thm-janw-cc-2})$ readily implies that
\begin{align*}
 R(x)&\geq  m +  r - |n|\\
& = \tfrac{c^2}{16}+(\kappa-1)^2 + (\kappa-1)\tfrac{1+B}{1-B}- \tfrac{(1-AB)^2}{16(A-B)^2}c^2 \\
& \quad \quad -\left|(\kappa-1)\tfrac{(A+B)}{2(A-B)}c +\tfrac{(1+B)^2(1+A)}{4(1-B)(A-B)}c\right| \\
&= (\kappa-1)^2+(\kappa-1)\tfrac{(1+B)}{1-B}- \left|(\kappa-1)\tfrac{(A+B)}{2(A-B)}c +\tfrac{(1+B)^2(1+A)}{4(1-B)(A-B)}c\right|\\
& \quad \quad-\tfrac{(1-A^2)(1-B^2)}{16(A-B)^2}c^2\\
& \geq 0.
\end{align*}
Now considers the case  of the constraint $(\ref{eqn:thm-janw-cc-4})$, which is equivalent to $|n| < 2m$. Then the minimum of $R$ occurs at $x = - n/(2m)$, and
$(\ref{eqn:thm-janw-cc-3})$ yields
\begin{align*}
R(x) \geq  \tfrac{4mr-n^2}{4m} \geq 0.
\end{align*}

Evidently  $\Psi$ satisfies the hypothesis of Lemma $\ref{lem:miller-mocanu-1}$, and thus $\RM\; p(z) > 0$, that is,
\[
- \tfrac{ (1-A) - (1-B) u_p(z)}{ (1+A) - (1+B) u_p(z)} \prec \tfrac{1+z}{1-z}.
\]
Hence there exists an analytic self-map  $w$ of $\mathbb{D}$ with $w(0)=0$ such that
\[
- \tfrac{ (1-A) - (1-B) u_p(z)}{ (1+A) - (1+B) u_p(z)} = \tfrac{1+w(z)}{1-w(z)},
\]
which implies that $u_p(z) \prec (1+ A z)/(1+B z).$
\end{proof}

Theorem $\ref{thm:-janw-cc}$  gives rise to simple conditions on $c$ and $\kappa$ to ensure $u_p(z)$ maps $\mathbb{D}$ into a half-plane.

\begin{corollary}\label{cor:-janw-cc-1}
Let $c \leq 0$ and $2\kappa \geq  2+c^2$. Then $\RM u_p(z) > c/(c-1).$
\end{corollary}
\begin{proof}
Choose $A= - (c+1)/(c-1)$, and $B=-1$ in Theorem $\ref{thm:-janw-cc}$. Then both the conditions $(\ref{eqn:thm-janw-cc-0})$ and $(\ref{eqn:thm-janw-cc-1})$ are equivalent  to $\kappa \geq 1$ which clearly holds for $\kappa \geq 1+ c^2/2$.  The proof will complete if the hypothesis $(\ref{eqn:thm-janw-cc-2})$  holds, i.e.,
\begin{align}\label{eqn:cor-1:-janw-cc-2}
(\kappa-1)^2  \geq \tfrac{1}{2}(\kappa-1)c^2.
\end{align}
Since $\kappa \geq 1+c^2/2$, it follows that
 \begin{align*}
(\kappa-1)^2-\tfrac{1}{2}(\kappa-1)c^2=(\kappa-1) \left(\kappa -1-\tfrac{c^2}{2}\right) \geq 0,
\end{align*}
which establishes $(\ref{eqn:cor-1:-janw-cc-2})$.
\end{proof}

\begin{corollary}\label{cor:-janw-cc-3}
Let $c$, $\kappa$  be real such that
\begin{align*}
\kappa \geq \left\{ \begin{array}{lll}
1, &  c \leq 0\\
1+\frac{c}{2}, & c \geq 0.
\end{array}\right.
\end{align*}
Then $ \RM u_p(z) > 1/2$.
\end{corollary}
\begin{proof}
Put $A=0$ and $B=-1$ in Theorem $\ref{thm:-janw-cc}$.  The condition
$(\ref{eqn:thm-janw-cc-0})$ reduces to  $\kappa \geq 1$, which holds in all cases.
It is sufficient to establish conditions
 $(\ref{eqn:thm-janw-cc-1})$ and $(\ref{eqn:thm-janw-cc-2})$, or equivalently,
\begin{align}\label{eqn:cor-3:-janw-cc-1}
4(\kappa-1)-c \geq 0,
\end{align} and
\begin{align}\label{eqn:cor-3:-janw-cc-2}
(\kappa-1)^2 - \tfrac{1}{2}(\kappa-1)c \geq 0.
\end{align}

For the case when  $c \leq 0$, both the inequality $(\ref{eqn:cor-3:-janw-cc-1})$ and $(\ref{eqn:cor-3:-janw-cc-2})$ hold
as $\kappa \geq 1$.

Finally it is readily established for $c \geq 0$ and $\kappa -1\geq  c/2$ that
$4(\kappa-1)-c \geq c \geq 0,$ and
$(\kappa-1)^2 - \tfrac{1}{2}(\kappa-1)c \geq(\kappa-1)(\kappa-1-\tfrac{c}{2})\geq 0.$
\end{proof}

It is known that for $b=2$ and $c=\pm 1$, the generalized Bessel functions $ u_{p,2,1}(z)=j_p(z)$ and $ u_{p,2,-1}(z)=i_p(z)$
respectively gives the spherical Bessel and modified spherical Bessel functions. This specific choice of $b$ and $c$,
Corollary $\ref{cor:-janw-cc-3}$ yield $\RM(i_p(z))>1/2$ for $p \geq -1/2$, and $\RM(j_p(z))>1/2$, for $p \geq 0$. Since $i_p'(0)= 1/(4p+6)$ for $p\geq -1/2$, following inequalities can be obtain with the aid of results in \cite{McCarty}.
\begin{corollary} For $p \geq -1/2$, the modified spherical Bessel functions $i_p$ satisfy the following inequalities.
\begin{align}
|i_p(z)| &\leq \frac{ 4p+6+|z|}{2(2p+3)(1-|z|^2)},\\
\RM(i_p(z) &\geq \frac{p+6+|z|}{4p+6+2|z|+2(2p+3)|z|^2},\\
|i_p'(z)| &\leq \frac{2 \RM(i_p(z)-1}{2(1-|z|^2)}\times \frac{|z|^2+4(2p+3)|z|+1}{(2p+3)|z|^2+|z|+(2p+3)}.
\end{align}
\end{corollary}
Next theorem gives the sufficient condition for close-to-convexity when $B \geq 3- 2 \sqrt{2}.$
\begin{theorem}\label{thm:janw-cc-2}
Let $3- 2 \sqrt{2} \leq B < A\leq1 $
  and $c$, $\kappa  \in \mathbb{R}$ satisfy
 \begin{align}\label{eqn:thm-janw-cc-2-0}
\kappa-1 \geq \max \left\{0,  \tfrac{(1+B)(1+A)}{4(A-B)}\left|c\right|\right\}.
\end{align}
Suppose $A$, $B$, $\kappa$ and $c$ satisfy either the inequality
\begin{align}\label{eqn:thm-2-janw-cc-1}\nonumber
&(\kappa-1)^2+ 16(\kappa-1)\tfrac{B(1-B)}{(1+B)^3}- \left| \tfrac{(\kappa-1)(A+B) }{2(A-B)} c  + \tfrac{4B(1-B^2)(1+A) }{(1+B)^3(A-B)} c\right| \\
& \quad \quad \geq  \tfrac{(1-A^2) (1-B^2)}{16(A-B)^2} c^2
\end{align}
whenever
\begin{align}\label{eqn:thm-2-janw-cc-2}
\left|  (\kappa-1)(1+B)^3(A+B) c+8B(1-B^2)(1+A) c \right| \geq \tfrac{c^2}{4}(A-B)(1+B)^3 ,
\end{align}
or the inequality
\begin{align}\label{eqn:thm-2-janw-cc-3}\nonumber
&\left((\kappa-1)  \tfrac{(A+B)}{2(A-B)}c + \tfrac{4B(1-B^2)((1+A)}{(1+B)^3(A-B)}c \right)^2\\
 &\leq \tfrac{c^2}{4} \left((\kappa-1)^2 +16(\kappa-1)\tfrac{B(1-B)}{(1+B)^3}  -\tfrac{(1-AB)^2}{16(A-B)^2}c^2 \right)
\end{align}
whenever
\begin{align}\label{eqn:thm-2-janw-cc-4}
&\left| \left((\kappa-1)(A+B)(1+B)^3+8B(1-B^2)(1+A)\right) c \right| < \tfrac{c^2}{4}(A-B)(1+B)^3 .
\end{align}
If $(1+B)u_p (z) \neq (1+A)$,  then
$u_p (z) \in \mathcal{P}[A,B]$.
\end{theorem}
\begin{proof}
 First, proceed
 similar to the proof of Theorem $\ref{thm:-janw-cc}$ and derive the expression of $\RM \Psi( i\rho, \sigma, \mu+ i \nu; z)$ as given in $(\ref{eqn:thm-1-re-psi})$ . Now
for $\sigma \leq -(1+\rho^2)/2$, $\rho \in \mathbb{R}$, and $B \geq 3- 2 \sqrt{2}$,
\begin{align*}
\tfrac{ 2( 1-B^2)}{(1-B)^2+ (1+B)^2 \rho^2} \sigma^2
 &\geq \tfrac{ 2( 1-B^2)}{(1-B)^2+ (1+B)^2 \rho^2} \tfrac{(1+\rho^2)^2}{4}
\geq \dfrac{8 B (1-B)}{(1+B)^3},
\end{align*}
and then with $z = x+iy \in \mathbb{D}$, and $\mu+\sigma<0$, it follows that
\begin{align*}
&\RM\Psi( i\rho, \sigma, \mu+ i \nu; z)\\
&\leq -\tfrac{1}{2}(\kappa-1) (1+\rho^2) - \tfrac{(1+B) (1+A)\rho^2}{8(A-B)} cx
- \tfrac{\rho (1-AB)}{4(A-B)} cy\\
& \quad \quad + \tfrac{(1-B) (1-A)}{8(A-B)} cx -  \tfrac{8 B (1-B)}{(1+B)^3}\\
&= p_2 \rho^2 + q_2 \rho + r_2: = Q_1(\rho),
\end{align*}
where
\begin{align*}
p_2&= -\tfrac{1}{2}(\kappa-1)-\tfrac{(1+B)(1+A)}{8(A-B)}cx, \hspace{2.5in} \\
q_2&=-\tfrac{(1-AB)cy}{4(A-B)},\\
r_2 &= -\tfrac{1}{2}(\kappa-1)+\tfrac{(1-B)(1-A)}{8(A-B)}cx-\tfrac{8B(1-B)}{(1+B)^3}.
\end{align*}

Observe that the inequality
$(\ref{eqn:thm-janw-cc-2-0})$ implies that $p_2 <0$. Thus $Q_1(\rho) <0$ for all $\rho \in \mathbb{R}$ provided $q_{2}^{2} \leq 4p_{2}r_2$, that is, for $|x|, |y|<1$,
\begin{align*}
&\tfrac{(1-AB)^2}{16(A-B)^2} c^2y^2 \\
&\leq   \left( (\kappa-1)+\tfrac{(1+B) (1+A)}{4(A-B)} cx\right) \left((\kappa-1)-\tfrac{(1-B) (1-A)}{4(A-B)} cx + \tfrac{16 B (1-B)}{(1+B)^3}\right).
 \end{align*}
With $y^2 < 1- x^2$, it is enough to show for $|x|<1$,
 \begin{align*}
&\tfrac{(1-AB)^2}{16(A-B)^2} c^2(1-x^2)\\
&\leq   \left( (\kappa-1)+\tfrac{(1+B) (1+A)}{4(A-B)} cx\right) \left((\kappa-1)-\tfrac{(1-B) (1-A)}{4(A-B)} cx + \tfrac{16 B (1-B)}{(1+B)^3}\right),
 \end{align*}
which is equivalent to
\begin{align}\label{eqn:thm-2-x}
R_1(x) := m_1 x^2 + n_1 x+ r_1 \geq 0,
\end{align}
where
\begin{align*}
m_1  &:=\tfrac{c^2}{16},\\
n_1 &:= \left( (\kappa-1)\left( \tfrac{c (A+B)}{2(A-B)} \right) +
 \tfrac{4 B (1-B^2)(1+A)}{(A-B)(1+B)^3}c\right),\\
r_1 &:= (c-1)^2 + (c-1) \tfrac{16 B (1-B)}{(1+B)^3}- \tfrac{a^2(1-AB)^2}{(A-B)^2}.
\end{align*}

If $(\ref{eqn:thm-2-janw-cc-2})$ holds, then $|n_1| \geq 2|m_1|$. Since $R_1$ is increasing, then $R_1(x) \geq m_1 + r_1 - |n_1|$,  which is nonnegative from $(\ref{eqn:thm-2-janw-cc-1})$. On the other hand, if
$(\ref{eqn:thm-2-janw-cc-4})$ holds, then $|n_1| < 2|m_1|$,  $R_1(x) \geq  (4 m_1 r_1 - n_1^2)/ 4 m_1$,  and
$(\ref{eqn:thm-2-janw-cc-3})$ implies $R_1(x) \geq 0$. Either case establishes $(\ref{eqn:thm-2-x})$.
\end{proof}

\begin{theorem}\label{thm:-janw-cc-3}
Let $-1 \leq B \leq 3- 2 \sqrt{2} \approx 0.171573$. Suppose
 $B < A \leq 1$, $c, \kappa \in \mathbb{R}$ with $c \neq 0$ and satisfying
 \begin{align*}
\kappa \geq \max \left\{0,\tfrac{(1+B)(1+A)}{4(A-B)} \left|c\right|\right\}.
\end{align*}
Further let $A$, $B$, $\kappa$ and $c$ satisfy either
 \begin{align*}
&\kappa^2+\kappa\tfrac{1+B}{1-B}- \left| \kappa\tfrac{(A+B)}{2(A-B)}c + \tfrac{(1+B)^2(1+A)}{4(1-B)(A-B)}c \right|\geq  \tfrac{(1-A^2) (1-B^2)}{16(A-B)^2} c^2
\end{align*}
whenever
 \begin{align*}
\left| 2 \kappa(1-B)(A+B)c + (1+B)^2 (1+A) c) \right| \geq \tfrac{1}{2}(A-B)(1-B)c^2,
\end{align*}
or the inequality
 \begin{align*}
\left(\kappa\tfrac{(A+B)}{2(A-B)}+\tfrac{(1+B)^2(1+A)}{4(1-B)(A-B)}c\right)^2
 \leq \tfrac{c^2}{4}\left(\kappa^2+\tfrac{\kappa(1+B)}{1-B}-\tfrac{(1-AB)^2}{16(A-B)^2}c^2\right)
\end{align*}
when
 \begin{align*}
\left| 2 \kappa(1-B)(A+B)c+(1+B)^2(1+A)c \right| <\tfrac{1}{2}(A-B)(1-B)c^2.
\end{align*}
If $(1+B)u_p(z) \neq (1+A)$,  then
$(-4\kappa/c)u'_p(z) \in \mathcal{P}[A,B]$.
\end{theorem}
\begin{theorem}\label{thm:-janw-cc-4}
Let  $ 3- 2 \sqrt{2} < B  < A \leq 1$. Suppose $c, \kappa \in \mathbb{R}$, $a \neq 0$, such that
 \begin{align*}
 \kappa \geq \max \left\{0,\tfrac{(1+B)(1+A)}{4(A-B)} \left|c\right|\right\}.
\end{align*}
Suppose $A$, $B$, $\kappa$ and $c$ satisfy either
 \begin{align*}
\kappa^2+16\kappa\tfrac{B(1-B)}{(1+B)^3}- \left| \kappa\tfrac{(A+B)}{2(A-B)}c + \tfrac{4B(1-B^2)(1+A)}{(1+B)^3(A-B)}c \right|   \geq  \tfrac{(1-A^2) (1-B^2)}{16(A-B)^2} c^2
\end{align*}
whenever
 \begin{align*}
\left|  \kappa(1+B)^3(A+B)c + 8B(1-B^2) (1+A) c) \right| \geq \tfrac{c^2}{4}(A-B)(1+B)^3,
\end{align*}
or the inequality
 \begin{align*}
\left(\kappa\tfrac{(A+B)}{2(A-B)}c+\tfrac{4B(1-B^2)(1+A)}{(1+B)^3(A-B)}\right)^2
 \leq \tfrac{c^2}{4}\left(\kappa^2+\tfrac{16\kappa(B(1-B)}{(1+B)^3}-\tfrac{(1-AB)^2}{16(A-B)^2}c^2\right)
\end{align*}
when
 \begin{align*}
\left| 2 \kappa(1+B)^3(A+B)c+8B(1-B^2)(1+A)c \right| <\tfrac{c^2}{4}(A-B)(1+B)^3.
\end{align*}
If $(1+B)u_p(z) \neq (1+A)$,  then
$(-4\kappa/c)u'_p(z) \in \mathcal{P}[A,B]$.
\end{theorem}

\begin{corollary}
Let $c \leq -1$, and
 \[\kappa \geq  \max \left\{\tfrac{c(c+1)}{2}, \tfrac{c}{2(c+1)}\right\}.\]
Then $(-4\kappa/c)(u_p(z)-1)$ is close-to-convex of order $(c+1)/c$ with respect to the identity function.
\end{corollary}


\begin{corollary}
Let $c$  be a nonzero real number, and
$\kappa \geq |c|/2$. Then \[ \RM (-4\kappa/c) u'_p(z) > 1/2.\]
\end{corollary}

\section{Janowski starlikeness of generalized Bessel functions}\label{section-4}

This section contributes to find conditions to ensure a normalized and generalized Bessel functions $z u_p(z)$ in the class of
Janowski starlike functions. For this purpose, first  sufficient conditions for $u_p(z)$ to be Janowski convex is determined, and then an application of relation $( \ref{eqn:kumar-hypr-recur-1})$ yields conditions for $z u_p(z) \in \mathcal{S}^*[A, B].$

 \begin{theorem}\label{thm:jan-convex}
Let $c, \kappa \in \mathbb{R}$ be such that $(A-B)u'_p(z) \neq (1+B) z u''_p(z)$, $ -1 \leq B \leq 0 < A \leq 1$.
 Suppose
 \begin{align} \label{eqn:thm-jan-conv-1}
 \kappa(1+B)\geq   \tfrac{(1+B)^2}{4(A-B)}\left|c\right|-(1+A-B).
 \end{align}
Further let $A$, $B$, $\kappa$ and $c$ satisfy
\begin{align}\label{eqn:thm-jan-conv-2}
&(1+A-B+\kappa(1+B))(1-A+B+\kappa(1-B))\\
&\quad \quad \geq \tfrac{(1-B^2)^2}{16(A-B)^2}c^2- \left|\tfrac{B-(A-B)(1+B^2)+(1-B^2)B\kappa}{2(A-B)}c\right|
\end{align}

If $0 \notin u'_p(\mathbb{D})$, $0 \notin u''_p(\mathbb{D})$, then
\[1 + \tfrac{z u''_p(z)}{u'_p(z)} \prec \tfrac{1+Az}{1+Bz}.\]
 \end{theorem}

\begin{proof}
Define an analytic function  $p : \mathbb{D} \to \mathbb{C}$ by
\[
p(z): = \tfrac{ (A-B) u'_p(z) + (1-B) z u''_p(z)}{(A-B) u'_p(z) - (1+B) z u''_p(z)}.
\]
Then
\begin{align}\label{eqn:thm-jan-conv-6}
 \tfrac{ z u''_p(z)}{ u'_p(z)}  =  \tfrac{(A-B) (p(z)-1)}{(p(z)+1)+ B (p(z)-1)},
 \end{align}
and
\begin{align}\label{eqn:thm-jan-conv-6-1}\nonumber
\tfrac{z^2 u'''_p(z)+zu''_p(z)}{ zu''_p(z)} - \tfrac{  zu''_p(z) }{ u'_p(z)}
&=  \tfrac{z p'(z)}{(p(z)-1)}-  \tfrac{(1+B) z p'(z)}{(p(z)+1)+ B (p(z)-1)}\\
&= \tfrac{z p'(z) \left( (p(z) +1)+ B (p(z)-1)- (1+B) (p(z)-1)\right)}{ (p(z)-1)((p(z)+1)+ B(p(z)-1))}.
\end{align}
A rearrangement of $(\ref{eqn:thm-jan-conv-6-1})$ yields
\[ \tfrac{ z u'''_p(z) }{ u''_p(z)} = \tfrac{2 z p'(z)}{ (p(z)-1) ((p(z)+1)+ B (p(z)-1))}-1+ \tfrac{  zu''_p(z) }{ u'_p(z)}. \]
Thus,
\begin{align}\label{eqn:thm-jan-conv-7}\nonumber
&\left(\tfrac{ z u'''_p (z)}{ u''_p(z)}\right)\; \left( \tfrac{ zu''_p(z) }{u'_p(z)} \right)\\
&= \tfrac{2(A-B)(p(z)-1)  z p'(z)}{ (p(z)-1) \big((p(z)+1)+ B (p(z)-1)\big)^2}
 - \tfrac{(A-B) (p(z)-1)}{(p(z)+1)+ B (p(z)-1)}+ \tfrac{(A-B)^2 (p(z)-1)^2}{\left((p(z)+1)+ B (p(z)-1)\right)^2}.
\end{align}
Now a differentiation of $(\ref{eqn:kumar-hypr-ode})$  leads to
\begin{align*}
4z^2u'''_p(z)+4(\kappa+1)zu''_p(z)+czu'_p(z) = 0,
\end{align*}
which give
\begin{align}\label{eqn:thm-jan-conv-9}
\left(\tfrac{ z u'''_p(z) }{ u''_p(z)}\right)\; \left( \tfrac{ zu''_p(z) }{u'_p(z)} \right)
+ (\kappa+1)\tfrac{ zu''_p(z) }{u'_p(z)}+\tfrac{c}{4}  z= 0.
\end{align}
Using $(\ref{eqn:thm-jan-conv-6})$ and $(\ref{eqn:thm-jan-conv-7})$,
 $(\ref{eqn:thm-jan-conv-9})$  yields
\begin{align*}
 \tfrac{2(A-B)z p'(z)}{ \big((p(z)+1)+ B (p(z)-1)\big)^2}  + \tfrac{(A-B)^2 (p(z)-1)^2}{\big((p(z)+1)+ B (p(z)-1)\big)^2} +  \tfrac{(A-B) (p(z)-1)\kappa}{(p(z)+1)+B (p(z)-1)}+\tfrac{c}{4}  z =0,
\end{align*}
equivalently
\begin{align}\label{eqn:thm-jan-conv-10} \nonumber
& z p'(z)+ \left( \tfrac{A-B}{2} + \tfrac{\kappa(1+B)}{2}  + \tfrac{ cz(1+B)^2 } { 8( A-B)} \right) (p(z))^2 - \left(  A- B + \kappa B - \tfrac{ c(1-B^2) } { 4( A-B)}z \right) p(z)\\
& + \left( \tfrac{A-B}{2} - \tfrac{\kappa(1-B)}{2}  + \tfrac{ cz(1-B)^2 } { 8( A-B)} \right)= 0.
\end{align}

Define,
\[\Psi( p(z), z p'(z), z) := z p'(z)+ F_{1} (p(z))^2 + F_{2}\; p(z) + F_{3},\]
where
\begin{align*}
F_1&=  \tfrac{(A-B)}{2} + \tfrac{\kappa(1+B)}{2}  + \tfrac{ cz(1+B)^2 } { 8( A-B)},\\
F_2& =  - (A-B) - \kappa B  + \tfrac{ c(1-B^2) } {4 ( A-B)}z, \\
F_3 &=  \tfrac{(A-B)}{2} - \tfrac{\kappa(1-B)}{2}  + \tfrac{ cz(1-B)^2 } { 8( A-B)}.
\end{align*}
 Thus, $(\ref{eqn:thm-jan-conv-10})$ yields $\Psi( p(z), z p'(z), z)  \in  \Omega =\{0\}$. Now with $ z= x+iy \in \mathbb{D}$, let
\begin{align*}
G_1 &:= \RM (F_1)=  \tfrac{A-B}{2} + \tfrac{\kappa(1+B)}{2}  + \tfrac{ cx (1+B)^2 } { 8( A-B)}\\
   & = \tfrac{1}{2} \left( A-B+ \kappa(1+B) + \tfrac{cx (1+B)^2}{4(A-B)}\right) ,\\
G_2 &:= \RM ( i F_2)  = -\tfrac{c(1-B^2)}{4(A-B)}y,  \\
G_3 &:= \RM (F_3)= \tfrac{A-B}{2} - \kappa\tfrac{1-B}{2}  + \tfrac{ cx(1-B)^2 } {8 ( A-B)}\\
 &=\tfrac{1}{2} \left( A-B- \kappa(1-B) + \tfrac{c (1-B)^2}{4(A-B)}x\right).
\end{align*}

For  $ \sigma \leq - (1+ \rho^2)/2$, $\rho \in \mathbb{R}$,
\begin{align*}
\RM  \Psi( i \rho, \sigma, z)
 &= \sigma - G_1 \rho^2 + G_2 \rho + G_3\\
 &\leq -\tfrac{1+2 G_1}{2}  \rho^2 +  G_2\; \rho +\tfrac{ 2 G_3 +1}{2} := Q (\rho).
\end{align*}
 Note that condition $(\ref{eqn:thm-jan-conv-1})$ implies  $(1+2 G_1)/{2} >0$. In this case, $Q$ has a maximum at $\rho=G_2/(1+2G_1)$. Thus $Q (\rho) <0$ for all real $\rho$ provided
\begin{align*}
G_2^2 \leq (1+ 2 G_1) (1- 2 G_3 ), \quad |x|, |y|<1.
\end{align*}
Since $y^2 < 1- x^2$, it is left to show that
\begin{align*}
&\tfrac{(1-B^2)^2}{16(A-B)^2}c^2 (1-x^2)\\
&\leq \left( 1+A-B+ \kappa(1+B) +\tfrac{c(1+B)^2}{4(A-B)}x\right)\left(1- A+B-\kappa(-1+B)\right.\\
&\quad \quad \quad \left.- \tfrac{c(1-B)^2}{4(A-B)}x \right),
\end{align*}
$|x|<1$. The above inequality is equivalent to
\begin{align}\label{eqn:thm-1-h}
H(x) : =  h_2(A,B) x + h_3(A,B)\geq 0,
\end{align}
where
\begin{align*}
h_2(A, B)  & = -\tfrac{(B-(A-B)(B^2+1)+(1-B^2)B\kappa)c}{2(A-B)},\\
h_3(A, B)  &=\left(1+A-B+\kappa (1+B)\right)\left(1-A+B-\kappa (B-1)\right)-\tfrac{(1-B^2)^2}{16(A-B)^2}c^2.
 \end{align*}

Since $|x|<1$,  the left-hand side of the inequality $(\ref{eqn:thm-1-h})$ satisfy 
\begin{align*}
h_2(A,B)x+h_3(A,B) \geq -|h_2(A,B)|+h_3(A,B). 
\end{align*}
Now it is evident from $(\ref{eqn:thm-jan-conv-2})$ that $H(x) \geq 0$ which establish the inequality 
$(\ref{eqn:thm-1-h})$.
 
 Thus $\Psi$ satisfies the hypothesis of Lemma $\ref{lem:miller-mocanu-1}$, and hence $\RM\; p(z) > 0$, or equivalently
 \begin{align*}
\tfrac{ (A-B) u'_p + (1-B) z u''_p}{(A-B) u'_p - (1+B) z u''_p} \prec \tfrac{1+z}{1-z}.
 \end{align*}
By definition of subordination, there exists an analytic self-map $w$ of $\mathbb{D}$ with $w(0)=0$ and
\begin{align*}
\tfrac{ (A-B) u'_p(z) + (1-B) z u''_p(z)}{(A-B) u'_p(z) - (1+B) z u''_p(z)} = \tfrac{1+w(z)}{1-w(z)}.
 \end{align*}
 A simple computation shows that
 \begin{align*}
1+ \tfrac{z u''_p(z)}{u'_p(z)} = \tfrac{1+Aw(z)}{1+Bw(z)},
 \end{align*}
and hence
 \[1+ \tfrac{z u''_p(z)}{u'_p(z)} \prec \tfrac{1+Az}{1+Bz}.\qedhere\]
\end{proof}
The relation  $(\ref{eqn:kumar-hypr-recur-1})$ also shows that
\begin{align*}
\tfrac{ z \; (z u_p(z)))'}{z u_p(z)} = 1+ \tfrac{ zu''_{p-1}(z)}{u'_{p-1}(z)}.
\end{align*}
Together with Theorem $\ref{thm:jan-convex}$, it immediately yields the following result for $z u_p(z)) \in \mathcal{S}^\ast[A,B]$.
\begin{theorem}\label{thm:jan-starlike}
Let $c$ and $\kappa$ be  real numbers such that $(A-B)u'_{p-1}(z) \neq (1+B) z u''_{p-1}(z)$, $ -1 \leq B < A\leq 1$.
 Suppose
 \begin{align}\label{eqn:thm-jan-star-1}
 \kappa(1+B)\geq  \tfrac{(1+B)^2}{4(A-B)}\left|c\right|-(A-2B).
 \end{align}
Further let $A$, $B$, $\kappa$ and $c$ satisfy 
 \begin{align*}
 (A-2B+\kappa(1+B))(2B-A+\kappa(1-B))\geq \tfrac{(1-B^2)^2}{16(A-B)}c^2+\left|\tfrac{B^3-(A-B)(1+B^2)+(1-B^2)B \kappa}{2(A-B)}c\right|
\end{align*}
Then $z u_p(z) \in \mathcal{S}^\ast[A,B]$.
\end{theorem}

\end{document}